\documentclass[reqno, 12pt]{amsart}
\usepackage{amsmath,amssymb,tikz,stix2}

\newtheorem{thm}{Theorem}
\newtheorem{lem}[thm]{Lemma}

\newtheorem{definition}[thm]{Definition}
\newtheorem{conj}[thm]{Conjecture}

\theoremstyle{definition}
\newtheorem{example}[thm]{Example}

\newtheorem*{thma}{Theorem A}
\newtheorem*{thmb}{Theorem B}
\newtheorem*{thmc}{Theorem C}

\newcommand{\C}{{\mathbb C}}
\newcommand{\D}{{\mathbb D}}

\newcommand{\ind}{\int_\D}

\newcommand{\cn}{\C^n}

\newcommand{\bn}{{\mathbb B}_n}
\newcommand{\sn}{{\mathbb S}^n}
\newcommand{\ins}{\int_{\sn}}
\newcommand{\inb}{\int_{\bn}}

\newcommand{\R}{{\mathbb R}}

\begin{document}

\title[Embeddings between Bergman and Hardy spaces]
{Embedding and compact embedding between\\ Bergman and Hardy spaces}

\author{Guanlong Bao}
\address[Guanlong Bao]{Department of Mathematics, Shantou University,
Shantou, Guangdong 515821, China}
\email{glbao@stu.edu.cn}

\author{Pan Ma}
\address[Pan Ma]{School of Mathematics and Statistics, HNP-LAMA,
Central South University, Changsha, Hunan 410083, China}
\email{pan.ma@csu.edu.cn}

\author{Fugang Yan}
\address[Fugang Yan]{College of Mathematics and Statistics, Chongqing University, 
Chongqing 401331, China; and Key Laboratory of Nonlinear Analysis and its 
Applications (Chongqing University), Ministry of Education, 
Chongqing 401331, China}
\email{fugangyan@cqu.edu.cn}

\author{Kehe Zhu}
\address[Kehe Zhu]{Department of Mathematics and Statistics,
State University of New York, Albany, NY 12222, USA.}
\email{kzhu@albany.edu}

\thanks{Bao is supported by NNSF of China (Grant number 12271328).
Ma is supported by NNSF of China (Grant numbers 11801572 and 12171484),
the Natural Science Foundation of Hunan Province (Grant number 2023JJ20056),
the Science and Technology Innovation Program of Hunan Province (Grant number
2023RC3028), and Central South University Innovation-Driven Research
Programme (Grant number 2023CXQD032). Yan is supported by NNSF of
China (Grant number 12401150) and the Fundamental Research Funds for the Central Universities (2024CDJXY018).}


\begin{abstract}
For Hardy spaces and weighted Bergman spaces on the open unit ball in $\cn$, we
determine exactly when $A^p_\alpha\subset H^q$ or $H^p\subset A^q_\alpha$,
where $0<q<\infty$, $0<p<\infty$, and $-\infty<\alpha<\infty$. For each such
inclusion we also determine exactly when it is a compact embedding. Although
some special cases were known before, we are able to completely cover all possible
cases here. We also introduce a new notion called {\it tight fitting} and formulate
a conjecture in terms of it, which places several prominent known results about
contractive embeddings in the same framework.
\end{abstract}

\maketitle

\section{Introduction}

For $n\ge1$ let $\bn$ be the unit ball in $\cn$ and let $dv$ denote the Lebesgue volume
measure on $\bn$. If $\alpha\in\R=(-\infty,\infty)$ is a real parameter, we write
$$dv_\alpha(z)=(1-|z|^2)^\alpha\,dv(z).$$
It is easy to see that $dv_\alpha$ is finite if and only if $\alpha>-1$. Throughout the
paper, whenever $\alpha>-1$, we also normalize $dv_\alpha$ such that $v_\alpha(\bn)=1$.

For $0<p<\infty$ and $\alpha\in\R$ we use $A^p_\alpha$ to denote the space of all
holomorphic functions $f$ on $\bn$ such that the function $(1-|z|^2)^NR^Nf(z)$ belongs
to $L^p(\bn, dv_\alpha)$, where $N$ is any non-negative integer such that
$Np+\alpha>-1$ and $R$ is the so-called radial differential operator defined on
holomorphic functions by
$$Rf(z)=z_1\frac{\partial f}{\partial z_1}(z)+\cdots+z_n\frac{\partial f}{\partial z_n}(z).$$
It is known that the space $A^p_\alpha$ is independent of the choice of the non-negative
integer $N$. The spaces $A^p_\alpha$ were called weighted Bergman spaces and
systematically studied in \cite{ZZ}.

For $\alpha>-1$ and $p>0$ we can choose $N=0$ above, so that
$$A^p_\alpha=L^p(\bn, dv_\alpha)\cap H(\bn),$$
where $H(\bn)$ denotes the space of all holomorphic functions on $\bn$. In this
case, we will write
$$\|f\|_{p,\alpha}=\left[\inb|f(z)|^p\,dv_\alpha(z)\right]^{1/p},\qquad f\in A^p_\alpha.$$
These are standard weighted Bergman spaces.

When $\alpha\le-1$, we let $N$ be the smallest positive integer such that $Np+\alpha>-1$
and write
$$\|f\|_{p,\alpha}=|f(0)|+\left[\inb\left|(1-|z|^2)^NR^Nf(z)\right|^p\,dv_\alpha(z)
\right]^{1/p},\qquad f\in A^p_\alpha.$$
These are more general ``weighted Bergman spaces'', which are sometimes called
holomorphic Sobolev spaces in the literature.

If $0<p<\infty$ and $d\sigma$ is the normalized Lebesgue measure on the unit sphere
$\sn=\partial\bn$, then the Hardy space $H^p$ consists of all holomorphic functions $f$
on $\bn$ such that
$$\|f\|^p_{H^p}=\sup_{0<r<1}\ins|f(r\zeta)|^p\,d\sigma(\zeta)<\infty.$$
See \cite{R, Zhu1} for basic properties of Hardy spaces.

The purpose of this paper is to completely determine when we have an embedding
$H^p\subset A^q_\alpha$ or $A^p_\alpha\subset H^q$. By the closed-graph theorem,
whenever we have such an embedding, the inclusion map is actually a bounded linear
operator. We will also completely determine when such an embedding is compact.

At first glance some readers may feel that the problem seems elementary, as we
often think that Hardy spaces are smaller than Bergman spaces. This instinct is only
partially right. It depends on how you make comparisons: if you compare Hardy and
standard (unweighted) Bergman spaces with the same exponents, then of course
the problem is simple. But if the exponents are different and the weights for Bergman
spaces are arbitrary, then the comparison can sometimes be very difficult to make.

For example, a classical theorem of Carleman's states that, on the unit disc, the Hardy
space $H^1$ is contained in the Bergman space $A^2$, and the inclusion is contractive;
see \cite{C}. This turns out to be a very deep and important result in complex analysis.
In fact, the exponent $2$ is optimal: if $p>2$, then $H^1$ is no longer contained
in $A^p\,$!

Since the Hardy spaces $H^p$ and weighted Bergman spaces $A^p_\alpha$ are not
Banach spaces when $0<p<1$, we need to specify the meaning of a compact
embedding. Thus we say that an embedding $i: X\to Y$ is compact if $\|f_n\|_Y\to0$
as $n\to\infty$ whenever $\{f_n\}$ is a bounded sequence in $X$ such that
$f_n(z)\to0$ uniformly on compact subsets of $\bn$.

Our characterizations of embedding and compact embedding between Hardy and
weighted Bergman spaces will depend on many different conditions on $p$, $q$, and
$\alpha$. In the statements of all our results, we will always use $p$ for the exponent
on the left-hand side and $q$ for the exponent on the right-hand side. We organize
our main results into the following two theorems for the two cases
$p<q$ and $p\ge q$, respectively.

\begin{thma}
Suppose $0<p<q<\infty$ and $\alpha\in\R$. Then
\begin{enumerate}
\item[(a)] $A^p_\alpha\subset H^q$ if and only if
\begin{equation}
\frac{n+1+\alpha}p\le\frac nq.
\label{eq1}
\end{equation}
\item[(b)] $H^p\subset A^q_\alpha$ if and only if
\begin{equation}
\frac np\le\frac{n+1+\alpha}q.
\label{eq2}
\end{equation}
\item[(c)] Any embedding in (a) and (b) is compact if and only if strict inequality
holds for the parameters.
\end{enumerate}
\end{thma}

Several special cases of Theorem A above and Theorem B below are known,
sometimes in slightly different form. See Theorems 5.12 and 5.13 in \cite{BB},
Corollary 1.3 in \cite{K}, and Theorem 4.48 in \cite{Zhu1} for example. As far
as we know, the compactness of these embeddings has not been systematically
studied before, except in the context of Carleson measures in some special
cases (see \cite{PZ,P}). In the case of embeddings between two weighted
Bergman spaces, compactness was characterized in \cite{Zhu3}.

\begin{thmb}
Suppose $0<q\le p<\infty$ and $\alpha\in\R$.
\begin{enumerate}
\item[(a)] For $q\le2$ we have $A^p_\alpha\subset H^q$ if and only if $\alpha\le-1$.
For $q>2$ we have $A^p_\alpha\subset H^q$ if and only if $\alpha<-1$.
\item[(b)] For $q<2$ we have $H^p\subset A^q_\alpha$ if and only if $\alpha>-1$.
For $q\ge2$ we have $H^p\subset A^q_\alpha$ if and only if $\alpha\ge-1$.
\item[(c)] Any embedding in (a) and (b) is compact if and only if the inequality
involving $\alpha$ is strict.
\end{enumerate}
\end{thmb}

It is well known that $H^2=A^2_{-1}$; see \cite{ZZ} for example. As a consequence
of Theorems A and B, we see that the only chance for $H^p=A^q_\alpha$ is when
$p=q=2$ and $\alpha=-1$. Actually, at least in the case $p\ge1$, $H^p$ and
$A^p_\alpha$ are not even isomorphic with equivalent ``norms'' unless $p=2$.
See \cite{W} and our Lemma~\ref{10} later.

It is well known that
$$\lim_{\alpha\to-1^+}\inb|f(z)|^p\,dv_\alpha(z)=\|f\|^p_{H^p}$$
for $f\in H^p$. See \cite{Zhu2} for a proof in the one-dimensional case.
The proof for higher dimensions is similar. The limit above indicates that, in some sense,
$H^p$ can be thought of as $A^p_\alpha$ as $\alpha\to-1$. It is then natural for us
to compare Theorems A and B above to the previously-known results about embeddings
between weighted Bergman spaces (see \cite{ZZ, Zhu3} for example), which we state
as follows for the convenience of readers.

\begin{thmc}
For $0<p,q<\infty$ and $\alpha,\beta\in\R$ we have
\begin{enumerate}
\item[(a)] If $p\le q$, then $A^p_\alpha\subset A^q_\beta$ if and only if
\begin{equation}
\frac{n+1+\alpha}p\le\frac{n+1+\beta}q.
\label{eq3}
\end{equation}
\item[(b)] If $p>q$, then $A^p_\alpha\subset A^q_\beta$ if and only if
\begin{equation}
\frac{1+\alpha}p<\frac{1+\beta}q.
\label{eq4}
\end{equation}
\item[(c)] Embeddings in the case $p>q$ are always compact, while embeddings
for $p\le q$ are compact if and only if strict inequality in (\ref{eq3}) holds.
\end{enumerate}
\end{thmc}

It is clear that Theorem A is {\em formally} a consequence of part (a) in Theorem C,
when $A^p_{-1}$ (or $A^q_{-1}$) is replaced by $H^p$ (or $H^q$) and $\alpha$
(or $\beta$) is replaced by $-1$. But Theorem B is much more complicated.
There are several delicate differences between Theorem C and Theorems A/B.
In particular, there is a difference between how the exponent ranges are broken down,
namely, ``$p\le q$ and $p>q$" versus ``$p<q$ and $p\ge q$".

This paper is motivated by a colloquium talk given by Professor Congwen Liu at
Shantou University in the summer of 2023, where he talked about an open problem on
contractive embeddings. We are grateful to Professor Liu for many other discussions.

\section{Preliminaries on Bergman spaces}

Recall that
$$dv_\alpha(z)=(1-|z|^2)^\alpha\,dv(z)$$
for $\alpha\le-1$. For $\alpha>-1$ we choose a positive constant $c_\alpha$ such that
$$dv_\alpha(z)=c_\alpha(1-|z|^2)^\alpha\,dv(z)$$
is a probability measure on $\bn$. For $p>0$ and $\alpha>-1$ the space
$$A^p_\alpha=L^p(\bn, dv_\alpha)\cap H(\bn)$$
is traditionally called a weighted Bergman space (with standard weights). This definition
does not work for $\alpha\le-1$, because in this case the measure
$dv_\alpha$ is infinite and the only function $f\in H(\bn)$ satisfying
$$\inb|f(z)|^p(1-|z|^2)^\alpha\,dv(z)<\infty$$
is the zero function.

The definition of $A^p_\alpha$ for $\alpha\le-1$ must involve derivatives. One way is
to use radial derivatives as in the Introduction above. Another way is to use ordinary
partial derivatives as follows: for any $p\in(0,\infty)$ and $\alpha\in\R$ choose a
non-negative integer $N$ such that $Np+\alpha>-1$ and define $A^p_\alpha$ as the
space of $f\in H(\bn)$ such that the functions
$$(1-|z|^2)^N\partial^mf(z),\qquad |m|=N,$$
all belong to $L^p(\bn, dv_\alpha)$, where $m=(m_1,\cdots,m_n)\in{\mathbb Z}^n_+$
with $|m|=m_1+\cdots+m_n$. It is known that the space $A^p_\alpha$ is independent
of the choice of $N$ and it coincides with the $A^p_\alpha$ defined in the Introduction.
See \cite{ZZ} for all this and more.

We call the spaces $A^p_\alpha$ weighted Bergman spaces, which is a widely accepted
term when $\alpha>-1$. For $\alpha\le-1$, these spaces have been studied in the
literature before the paper \cite{ZZ} appeared, and they were sometimes called
holomorphic Sobolev spaces and sometimes called holomorphic Besov spaces.
See \cite{BB} and references there.

The family of spaces $A^p_\alpha$ includes many classical function spaces as special
cases. Other than the natural weighted Bergman spaces when $\alpha>-1$, we have
the following prominent examples: $A^2_{-1}=H^2$; $A^2_{-n}$ is the so-called
Drury-Arveson space which is extremely important in multivariable operator theory;
$A^2_{-(n+1)}$ is the so-called Dirichlet space whose reproducing kernel is given by
$K(z,w)=1-\log(1-\langle z,w\rangle)$; and $A^p_{-(n+1)}$ are the so-called diagonal
Besov spaces. This shows that our problem of determining $H^p\subset A^q_\alpha$
and $A^p_\alpha\subset H^q$ is meaningful, non-trivial, and interesting.
A lot of classical and important holomorphic function spaces are covered here.

For convenience, we single out two special cases of Theorem C and state
them as the following two lemmas, which we will have to use many times later.

\begin{lem}
Fix $\alpha\in\R$ and suppose $0<p_1<p_2<\infty$. Then $A^{p_2}_\alpha
\subset A^{p_1}_\alpha$, and the inclusion is proper and compact.
\label{1}
\end{lem}

This lemma, as well as the next one, is also easy to prove directly and is a good
exercise for student readers.

\begin{lem}
Fix $p>0$ and suppose $-\infty<\alpha<\beta<\infty$. Then
$A^p_\alpha\subset A^p_\beta$, and the inclusion is proper and compact.
\label{2}
\end{lem}

Functions in each Hardy space and each weighted Bergman space have certain
maximal growth rates. This information, which we state as the following lemma,
is often very useful when we try to determine if one space is contained in another.

\begin{lem}
Suppose $p>0$ and $\alpha>-(n+1)$.
\begin{itemize}
\item[(a)] If $f\in H^p$, then
$$|f(z)|\le\frac{\|f\|_{H^p}}{(1-|z|^2)^{n/p}}$$
for all $z\in\bn$, and the exponent $n/p$ above is best possible.
\item[(b)] If $f\in A^p_\alpha$, then
$$|f(z)|\le\frac {C\|f\|_{A^p_\alpha}}{(1-|z|^2)^{(n+1+\alpha)/p}}$$
for all $z\in\bn$, where the constant $C$ is independent of $f$ and can be taken
to be $1$ when $\alpha>-1$. Moreover, the exponent $(n+1+\alpha)/p$ is best possible.
\end{itemize}
\label{3}
\end{lem}

\begin{proof}
This is well known. See \cite{Zhu1, ZZ}.
\end{proof}

Finally in this section we include an integral estimate that is frequently used in the
theory of Bergman spaces. We will also need it later in this paper.

\begin{lem}
Suppose $\alpha>-1$ and $t\in\R$. Then
$$\inb\frac{(1-|w|^2)^\alpha}{|1-\langle z,w\rangle|^{n+1+\alpha+t}}\,dv(w)
\sim\begin{cases}\displaystyle\frac1{(1-|z|^2)^t}, & t>0\\[8pt] 1, & t<0\end{cases}$$
for $z\in\bn$.
\label{4}
\end{lem}

\begin{proof}
This is well known. See \cite{R, ZZ, Zhu1} for example.
\end{proof}

\section{The proof of Theorem A}

Taking $q_1=0$, $q_2>0$, and $s_1=0$ in Theorem 5.13 in \cite{BB} gives the following:
if $0<p_1<p_2<\infty$, then $H^{p_1}\subset A^{p_2}_{(np_2)/p_1-(n+1)}$. Using our
notation here we can restate this result as follows.

\begin{lem}
If $0<p<q<\infty$, then $H^p\subset A^q_\alpha$ for
$$\alpha=\frac{nq}p-(n+1)\qquad{\rm i.e.}\qquad\frac np=\frac{n+1+\alpha}q.$$
\label{5}
\end{lem}

We can now prove part (b) of Theorem A.

\begin{thm}
If $0<p<q<\infty$ and $\alpha\in\R$, then $H^p\subset A^q_\alpha$ if and only if
$$\frac np\le\frac{n+1+\alpha}q.$$
\label{6}
\end{thm}

\begin{proof}
To prove the ``if part", we choose $\beta\le\alpha$ such that
$$\frac np=\frac{n+1+\beta}q\le\frac{n+1+\alpha}q.$$
It then follows from Lemmas \ref{5} and \ref{2} that $H^p\subset A^q_\beta\subset
A^q_\alpha$.

To prove the ``only if part", we assume that $H^p\subset A^q_\alpha$. It is easy to see
that we must have $n+1+\alpha>0$; otherwise, functions in $A^q_\alpha$ can grow at
most logarithmically (see Theorem~22 of \cite{ZZ}), while functions in $H^p$ can grow
faster near the boundary. Under this extra assumption, Lemma~\ref{3} says that
$f\in H^p\subset A^q_\alpha$ must satisfy the pointwise estimate
$$|f(z)|\le\frac C{(1-|z|^2)^{(n+1+\alpha)/q}},\qquad z\in\bn;$$
and the {\em optimal} pointwise estimate for functions $f$ in $H^p$ is given by
$$|f(z)|\le\frac C{(1-|z|^2)^{n/p}},\qquad z\in\bn.$$
Thus we must have
$$\frac np\le\frac{n+1+\alpha}q.$$
This completes the proof of the theorem.
\end{proof}

Setting $q_2=s_2=0$ and $q_1>0$ in Theorem 5.13 in \cite{BB}, we obtain the
following: if $0<p_1<p_2<\infty$, then $A^{p_1}_{q_1-p_1s_1-1}\subset H^{p_2}$ for
$$\frac{n+q_1}{p_1}-\frac n{p_2}=s_1.$$
Again, we will restate this as follows using our notation.

\begin{lem}\label{7}
If $0<p<q<\infty$, then $A^p_\alpha\subset H^q$ for
$$\alpha=\frac{np}q-(n+1)\qquad{\rm i.e.}\qquad \frac{n+1+\alpha}p=\frac nq.$$
\end{lem}

We can now prove part (a) of Theorem A.

\begin{thm}\label{8}
If $0<p<q<\infty$ and $\alpha\in\R$, then $A^p_\alpha\subset H^q$ if and only if
$$\frac{n+1+\alpha}p\le\frac nq.$$
\end{thm}

\begin{proof}
To prove the ``if part", we choose some $\beta\ge\alpha$ such that
$$\frac{n+1+\alpha}p\le\frac{n+1+\beta}p=\frac nq.$$
It follows from Lemmas \ref{2} and \ref{7} that
$A^p_\alpha\subset A^p_\beta\subset H^q$.

To prove the ``only if part", let us assume that $A^p_\alpha\subset H^q$. Then,
by Lemma~\ref{3}, every function $f\in A^p_\alpha\subset H^q$ satisfies the
pointwise estimate
$$|f(z)|\le\frac C{(1-|z|^2)^{n/q}},\qquad z\in\bn.$$
Without loss of generality, we may assume that
$n+1+\alpha>0$; otherwise, the desired estimate for the parameters is trivial. Under
this extra assumption, Lemma~\ref{3} says that the {\em optimal} pointwise estimate
for functions $f\in A^p_\alpha$ is given by
$$|f(z)|\le\frac C{(1-|z|^2)^{(n+1+\alpha)/p}},\qquad z\in\bn.$$
Thus we must have
$$\frac{n+1+\alpha}p\le\frac nq.$$
This completes the proof of the theorem.
\end{proof}

The following result, together with Theorems~\ref{6} and \ref{8}, will complete
the proof of Theorem A.

\begin{thm}\label{9}
Suppose $0<p<q<\infty$ and $\alpha\in\R$. Then
\begin{enumerate}
\item[(a)] $H^p$ is compactly contained in $A^q_\alpha$ if and only if
$$\frac np<\frac{n+1+\alpha}q.$$
\item[(b)] $A^p_\alpha$ is compactly contained in $H^q$ if and only if
$$\frac{n+1+\alpha}p<\frac nq.$$
\end{enumerate}
\end{thm}

\begin{proof}
To prove part (a), we first assume that $n/p<(n+1+\alpha)/q$. Then we can find another
parameter $\beta$ such that
$$\frac np<\frac{n+1+\beta}q<\frac{n+1+\alpha}q.$$
We then have $H^p\subset A^q_\beta\subset A^q_\alpha$. Since $\beta<\alpha$,
it follows from Lemma~\ref{2} that $A^q_\beta$ is compactly contained in
$A^q_\alpha$. Thus $H^p$ is compactly contained in $A^q_\alpha$.

If $n/p>(n+1+\alpha)/q$, then $H^p$ is not contained in $A^q_\alpha$ at all.

If $n/p=(n+1+\alpha)/q$, we have $H^p\subset A^q_\alpha$. To show that this
inclusion is not compact, we consider functions of the form
$$f_a(z)=(1-|a|^2)^{n/p}/(1-\langle z,a\rangle)^{2n/p},\qquad z\in\bn,$$
where $a\in\bn$. It is clear that $\|f_a\|_{H^p}=1$ for all $a\in\bn$ and
$f_a(z)\to0$ pointwise as $|a|\to1^-$.
Note that
$$\frac np=\frac{n+1+\alpha}q<\frac{n+1+\alpha}p,$$
which yields $\alpha>-1$. Thus
\begin{align*}
\|f_a\|^q_{q,\alpha}&=\inb|f_a(z)|^q\,dv_\alpha(z)\\
&=(1-|a|^2)^{nq/p}\inb\frac{dv_\alpha(z)}{|1-\langle z,a\rangle|^{2nq/p}}\\
&=(1-|a|^2)^{n+1+\alpha}\inb\frac{dv_\alpha(z)}
{|1-\langle z,a\rangle|^{2(n+1+\alpha)}}\\
&=1
\end{align*}
for all $a\in\bn$. In particular, $\|f_a\|_{q,\alpha}\not\to0$ as $|a|\to1^-$. Thus
the inclusion $H^p\subset A^q_\alpha$ is not compact. This completes the proof of
part (a).

Part (b) is proved in exactly the same way. We omit the details.
\end{proof}

\section{The case $p=q$}

In this section we consider the containment relations between $H^p$ and $A^p_\alpha$.
As was mentioned earlier, we have $A^2_{-1}=H^2$. It turns out that there is a critical
difference between the cases $p<2$ and $p>2$.

The following lemma is related but not critical to our analysis. But it is of some
independent interest.

\begin{lem}
Let $p\ge1$ and $\alpha\in\R$. Then $A^p_\alpha$ is isomorphic (with equivalent norms)
to $H^p$ if and only if $p=2$.
\label{10}
\end{lem}

\begin{proof}
The result is well known for the open unit disc; see \cite{W} for example.

In higher dimensions, it was shown in \cite{W1,W2,W4} that $H^p$ of the unit ball is
isomorphic to $H^p$ of the unit disc, which is also isomorphic to $L^p[0,1]$ (see
\cite{Boas}). By \cite{ZZ}, the weighted Bergman space $A^p_\alpha(\bn)$ is isomorphic
to the ordinary Bergman space $A^p(\bn)$. It was shown in \cite{W3} that
$A^p(\bn)$ is isomorphic to $l^p$. It is well known that $L^p[0,1]$ and $l^p$
are NOT isomorphic unless $p=2$ (see \cite{W}). It follows that, for $p\ge1$ and
$\alpha\in\R$, the spaces $A^p_\alpha(\bn)$ and $H^p(\bn)$ are not isomorphic
unless $p=2$.
\end{proof}

We note that there are a number of embedding theorems in
Beatrous-Burbea \cite{BB} that are related to our problem here. They appear as
Theorems 5.12 and 5.13 in \cite{BB}. Parts (ii) and (iii) of Theorem 5.12 in \cite{BB}
imply the embeddings in (a) and (c) below. We will also address the compactness
issue here.

\begin{lem}
We have
\begin{itemize}
\item[(a)] $H^p\subset A^p_{-1}$ for $2<p<\infty$.
\item[(b)] $H^2=A^2_{-1}$.
\item[(c)] $A^p_{-1}\subset H^p$ for $0<p\le2$.
\end{itemize}
Furthermore, the embeddings in (a) and (c) are proper (i.e., not equal) and non-compact.
\label{11}
\end{lem}

\begin{proof}
In addition to \cite{BB}, the three relations above can also be found in \cite{Zhu1}.

To show that the embeddings in (a) and (c) are never compact here, we consider
functions of the form
$$f_a(z)=\left[\frac{(1-|a|^2)^n}{(1-\langle z,a\rangle)^{2n}}\right]^{1/p},
\qquad a\in\bn.$$
It is clear that $f_a(z)\to0$ uniformly for $z$ in any compact subset of $\bn$ as $|a|\to1$.

On the other hand, we have $\|f_a\|_{H^p}=1$ for all $a\in\bn$. We can also find
positive constants $c$ and $C$ such that $c\le\|f_a\|_{A^p_{-1}}\le C$ for all
$a\in\bn$. In fact, the smallest non-negative integer $N$ such that $Np+(-1)>-1$
is $N=1$, so by Lemma~\ref{4},
$$\|f_a\|^p_{A^p_{-1}}\sim(1-|a|^2)^n\inb\frac{(1-|z|^2)^{p-1}\,dv(z)}
{|1-\langle z,a\rangle|^{2n+p}}\sim1$$
for $a\in\bn$. This shows that the embeddings in both (a) and (c) are never compact.

There are several ways to show that the embeddings in (a) and (c) are proper. One way
is to use Lemma~\ref{10}, at least in the case $p\ge1$. Another way is to use the
well-known fact that $A^p_{-1}$ has an ``atomic decomposition'' in terms of
kernel functions (see \cite{ZZ} for example), while $H^p$ definitely does not
have such a decomposition unless $p=2$.
\end{proof}

\begin{lem}\label{12}
If $p>0$ and $\alpha<-1$, then $A^p_\alpha\subset H^p$ and the embedding is
compact.
\end{lem}

\begin{proof}
If $0<p\le2$ and $\alpha<-1$, then $A^p_\alpha\subset A^p_{-1}\subset H^p$. The
compactness of the embedding follows from Lemma~\ref{2}.

If $p>2$ (and for $0<p\le2$ as well) and $\alpha<-1$, then we can find a sufficiently
small $\varepsilon\in(0,p)$ such that
$$\frac{n+1+\alpha}{p-\varepsilon}\le\frac np.$$
Then by Theorem A, we have $A^p_\alpha\subset A^{p-\varepsilon}_\alpha\subset H^p$.
The compactness of the embedding follows from Lemma~\ref{1}.
\end{proof}

\begin{thm}\label{13}
For $0<p<\infty$ and $\alpha\in\R$ we have the following.
\begin{itemize}
\item[(a)] If $p\le2$, then $A^p_\alpha\subset H^p$ if and only if $\alpha\le-1$.
\item[(b)] If $p>2$, then $A^p_\alpha\subset H^p$ if and only if $\alpha<-1$.
\item[(c)] If $p\ge2$, then $H^p\subset A^p_\alpha$ if and only if $\alpha\ge-1$.
\item[(d)] If $p<2$, then $H^p\subset A^p_\alpha$ if and only if $\alpha>-1$.
\end{itemize}
Moreover, each of the possible embeddings above is compact if and only if strict inequality
occurs in the condition for $\alpha$.
\end{thm}

\begin{proof}
If $p\le2$ and $\alpha\le-1$, then by Lemmas \ref{2} and \ref{11}, we have
$A^p_\alpha\subset A^p_{-1}\subset H^p$, and the inclusion is compact if
and only if $\alpha<-1$. If $p\le2$ and $\alpha>-1$, then it is impossible for
$A^p_\alpha\subset H^p$, because of the well-known optimal growth rates
for functions in $A^p_\alpha$ and $H^p$ (see Lemma~\ref{3}). This proves part (a).

If $p>2$ and $\alpha<-1$, then by Lemma \ref{12} we have $A^p_\alpha\subset H^p$
and the embedding is always compact. If $p>2$ and $\alpha\ge-1$, then it is impossible
for $A^p_\alpha\subset H^p$. In fact, even $A^p_{-1}\subset H^p$ is impossible;
otherwise, we would have $A^p_{-1}=H^p$ by Lemma~\ref{11}, which contradicts
the well-known fact that $A^p_\alpha$ and $H^p$ are not isomorphic with equivalent
norms unless $p=2$ (see Lemma~\ref{10}). This proves part (b).

If $p\ge2$ and $\alpha\ge-1$,  it follows from Lemmas \ref{11} and \ref{2} that
$H^p\subset A^p_{-1}\subset A^p_\alpha$ (when $\alpha>-1$, this also follows
from integration by polar coordinates), and the embedding $H^p\subset A^p_\alpha$
is compact if and only if $\alpha>-1$. On the other hand, if $p\ge2$ and $\alpha<-1$,
then we cannot possibly have $H^p\subset A^p_\alpha$; otherwise, by Lemma \ref{12},
we would have $A^p_\alpha=H^p$, which is impossible (even for $p=2$ here, because
$\alpha<-1$). This proves part (c).

Finally, if $p<2$ and $\alpha>-1$, then $H^p\subset A^p_\alpha$ by integration in
polar coordinates. If $-1<\beta<\alpha$, then we also have $H^p\subset A^p_\beta
\subset A^p_\alpha$, so the embedding $H^p\subset A^p_\alpha$ is compact in
view of Lemma \ref{2}. On the other hand, if $p<2$ and $\alpha\le-1$, then
$A^p_\alpha\subset A^p_{-1}\subset H^p$ by Lemma \ref{11}, which shows that
$H^p\subset A^p_\alpha$ is impossible; otherwise, we would have $A^p_{-1}=H^p$,
a contradiction to Lemma \ref{11}. This proves part (d) and completes the proof
of the theorem.
\end{proof}

\section{The proof of Theorem B}

In this section we will complete the proof of Theorem B. The case $p=q$ was already
settled in the previous section, so we only need to address the case $q<p$ here,
namely, we will determine exactly when $H^p\subset A^q_\alpha$ or
$A^p_\alpha\subset H^q$, where $0<q<p<\infty$.

We will begin with possible embeddings $H^p\subset A^q_\alpha$ for $q<p$.
First observe that if $\alpha\le-(n+1)$, then the inclusion
$H^p\subset A^q_\alpha$ is impossible, because of known pointwise estimates
in \cite{ZZ} for functions in $A^q_\alpha$ (they cannot exceed a power of logarithmic
growth near the boundary). Therefore, the most interesting case here is when
$\alpha>-(n+1)$ or $n+1+\alpha>0$.

If $n+1+\alpha>0$, it follows from optimal pointwise estimates (see \cite{ZZ,Zhu1}
or Lemma~\ref{3} here)
that $H^p\subset A^q_\alpha$ implies
$$\frac np\le\frac{n+1+\alpha}q.$$
So we actually have
$$\alpha\ge-(n+1)+\frac{nq}p=-1-n\left[1-\frac qp\right]\in(-\infty,-1).$$
Unlike the situation in Section 3, this condition (derived from
the optimal pointwise estimates) is not sufficient for $H^p\subset A^q_\alpha$.
We will see that the correct condition is $\alpha>-1$ for $q<2$ and $\alpha\ge-1$
for $q\ge2$. In order to prove this, we will need a certain ``Carleson measure''-type
theorem.

For any $\zeta\in\sn$ we define a so-called admissible approach region at $\zeta$
as follows:
$$\Gamma(\zeta)=\{z\in\bn:|1-\langle z,\zeta\rangle|<1-|z|^2\}.$$
This is standard notation in the literature, although it can be confused with the
classical gamma function. For any $r\in(0,\infty)$ and any Lebesgue measurable
function $g$ on $\sn$ we define a function $A_r(g)$ on $\sn$ by
$$A_r(g)(\zeta)=\left[\int_{\Gamma(\zeta)}|g(z)|^r\,d\lambda(z)\right]^{\frac1r},
\qquad \zeta\in\sn,$$
where
$$d\lambda(z)=\frac{dv(z)}{(1-|z|^2)^{n+1}}$$
is the M\"obius invariant volume measure on $\bn$. We also define a function
$A_\infty(g)$ on $\sn$ by
$$A_\infty(g)(\zeta)=\sup_{z\in\Gamma(\zeta)}|g(z)|,\qquad \zeta\in\sn.$$

\begin{thm}\label{14}
Suppose $0<q<p<\infty$, $\mu$ is a positive Borel measure on $\bn$, $m$ is
a positive integer, and
$$g(z)=\frac{\mu(D(z,1))}{(1-|z|^2)^{qm+n}},\qquad z\in\bn,$$
where $D(z,1)$ is the Bergman metric ball centered at $z$ with radius $1$. Let
$$A(g)(\zeta)=\begin{cases} A_{2/(2-q)}(g)(\zeta),&q<2,\\[10pt]
A_\infty(g)(\zeta),&q\ge2,\end{cases}\qquad \zeta\in\sn.$$
Then the following conditions are equivalent.
\begin{itemize}
\item[(a)] There exists a positive constant $C$ such that
$$\left[\inb|R^mf(z)|^q\,d\mu(z)\right]^{\frac1q}\le C\|f\|_{H^p}$$
for all $f\in H^p$.
\item[(b)] The function $A(g)(\zeta)$ belongs to $L^{p/(p-q)}(\sn, d\sigma)$.
\end{itemize}
\end{thm}

\begin{proof}
A version of this theorem was first proved by Luecking in \cite{L} in the context of
real upper-half spaces, using mostly techniques from real and harmonic analysis. It was
widely believed and accepted that Luecking's results are still valid for the open unit disc
with some obvious modifications. But the situation for the higher dimensional unit ball
was much less clear, and people were certainly hesitant about directly quoting Luecking's
theorems for the unit ball. For example, when a special case of Luecking's theorem
was needed for the unit ball in \cite{Pau}, Pau worked out all details in an appendix at
the end of the paper.

In 1999, Arsenovi\'c \cite{A} proved a version of Luecking's theorem in the contex of
$\mathcal M$-harmonic functions on the open unit ball $\bn$. The standard holomorphic
Hardy space $H^p$ is contained in the $\mathcal M$-harmonic Hardy space
${\mathcal H}^p$. Therefore, half of our theorem here is a direct consequence of \cite{A}.
The other half of our theorem follows from the proofs in \cite{A}, when
$\mathcal M$-harmonic functions are replaced by holomorphic functions.
\end{proof}

The following technical result is probably known to experts in several complex variables.
But we could not find a reference, so we include a detailed proof here.

\begin{lem}\label{15}
Let $\zeta\in\sn$ and $t\in\R$. Then
$$\int_{\Gamma(\zeta)}\frac{dv(z)}{(1-|z|^2)^{n+1+t}}\sim\int_{\Gamma(\zeta)}
\frac{dv(z)}{|1-\langle z,\zeta\rangle|^{n+1+t}}.$$
Moreover, these integrals are finite if and only if $t<0$.
\end{lem}

\begin{proof}
For $z\in\Gamma(\zeta)$ we have
$$1-|z|^2\ge|1-\langle z,\zeta\rangle|\ge1-|\langle z,\zeta\rangle|
\ge1-|z|.$$
Thus
$$1-|z|^2\sim1-|z|\sim|1-\langle z,\zeta\rangle|$$
for $z\in\Gamma(\zeta)$. This shows that the two integrals in the lemma are comparable.

Let us write
$$I(t)=\int_{\Gamma(\zeta)}\frac{dv(z)}{|1-\langle z,\zeta\rangle|^{n+1+t}}$$
for $z\in\bn$. It is clear that $I(t)$ is independent of $\zeta$. If $t<0$, the finiteness
of $I(t)$ follows immediately from Lemma \ref{4}. The other direction is more complicated.

If $n=1$ and if we take $\zeta=1$, then
$$I(t)=\int_{\Gamma(1)}\frac{dA(z)}{|1-z|^{2+t}}.$$
It follows easily from integration in polar coordinates centered at $z=1$ (namely, 
$z=1+re^{i\theta}$) that $I(t)<\infty$ implies $t<0$.

When $n>1$, we write
$$I(t)=\inb\frac{\chi_{\Gamma(\zeta)}(z)\,dv(z)}{|1-\langle z,\zeta\rangle|^{n+1+t}},$$
where $\chi_{\Gamma(\zeta)}$ is the characteristic function of
$\Gamma(\zeta)$. By polar coordinates, we have
$$I(t)=2n\int_0^1r^{2n-1}\,dr\ins\frac{\chi_{\Gamma(\zeta)}(r\eta)\,d\sigma(\eta)}
{|1-r\langle\eta,\zeta\rangle|^{n+1+t}}.$$
For any fixed $r\in(0,1)$, we can write
$$\ins\frac{\chi_{\Gamma(\zeta)}(r\eta)\,d\sigma(\eta)}{|1-r\langle\eta,\zeta
\rangle|^{n+1+t}}=\ins f(\langle\eta,\zeta\rangle)\,d\sigma(\eta),$$
where
$$f(w)=\begin{cases}0, &|1-rw|\ge1-r^2,\\[8pt]
\frac1{|1-rw|^{n+1+t}},&|1-rw|<1-r^2.\end{cases}$$
By (1.13) of \cite{Zhu1}, we have for $n>1$ that
\begin{align*}
I(t)&=2n(n-1)\int_0^1r^{2n-1}\,dr\ind(1-|w|^2)^{n-2}f(w)\,dA(w)\\
&=2n(n-1)\int_0^1r^{2n-1}\,dr\int_{D_r}\frac{(1-|w|^2)^{n-2}\,dA(w)}{|1-rw|^{n+1+t}},
\end{align*}
where $dA$ is normalized area measure on the unit disc $\D$ and
$$D_r=\{w\in\D: |1-rw|<1-r^2\}.$$
Since $1-r\le|1-rw|<1-r^2$ for $w\in D_r$, we have
$$I(t)\sim\int_0^1\frac{r^{2n-1}}{(1-r)^{n+1+t}}\,dr\int_{D_r}(1-|w|^2)^{n-2}\,dA(w).$$

For $r\in(0,1)$ the set $D_r$ is the intersection of two discs:
$$D_r=\{w:|w|<1\}\bigcap\left\{w:\left|w-\frac1r\right|<\frac1r-r\right\}.$$
This intersection is between the two vertical lines $x=r$ and $x=1$.
See Figure \ref{Fig1} below.

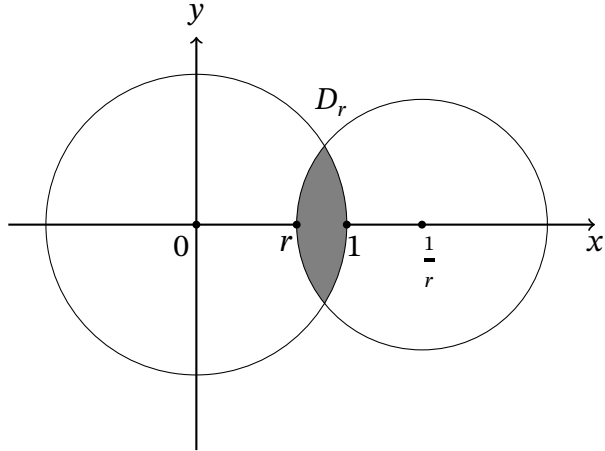
\begin{figure}[h]
\centering
\begin{tikzpicture}

\draw [thick,->](-2.5,0)-- (5.3,0);
\draw [thick,->](0,-3)-- (0,2.5);
\def\CA{(0,0) circle (2)}
\def\CB{(3,0) circle (5/3)}
\begin{scope}
\clip \CA;
\fill[gray,opacity=0.5] \CB;
\end{scope}
\draw \CA \CB;
\fill(2,0) circle (1pt);
\node[below] (P) at (2.1,0) {$1$};
\node[above] (P) at (1.8,1.3) {$D_r$};
\node[below] (P) at (3.1,0) {$\frac{1}{r}$}; \node[below] (P) at (1.2,0) {$r$}; \node[below] (P) at (5.3,0) {$x$}; \node[above] (P) at (0,2.5) {$y$}; \node[below] (P) at (-0.2,0) {$0$};
\fill(3,0) circle (1.5pt);
\fill(4/3,0) circle (1.5pt);
\fill(2,0) circle (1.5pt);
\fill(0,0) circle (1.5pt);
\end{tikzpicture}
\caption{The region $D_r$}
\label{Fig1}
\end{figure}

If we rotate the region $D_r$ around the origin by any angle to obtain another
region $D_r'$, then by symmetry, we have
$$\int_{D_r}(1-|w|^2)^{n-2}\,dA(w)=\int_{D_r'}(1-|w|^2)^{n-2}\,dA(w).$$
It is clear that it takes roughly $[2\pi/(1-r)]$-many such regions to cover
the annulus $(r+1)/2<|w|<1$. Therefore,
\begin{align*}
I(t)&\sim\int_0^1\frac{r^{2n-1}\,dr}{(1-r)^{n+t}}\int_{(r+1)/2<|w|<1}(1-|w|^2)^{n-2}
\,dA(w)\\
&\sim\int_0^1\frac{r^{2n-1}\,dr}{(1-r)^{t+1}}\sim\int_0^1\frac{dr}{(1-r)^{t+1}}.
\end{align*}
It is then clear that $I(t)<\infty$ if and only if $t<0$.
\end{proof}

We can now complete the proof of Theorem B, which we restate as follows.

\begin{thm}\label{16}
Suppose $0<q\le p<\infty$ and $\alpha\in\R$.
\begin{enumerate}
\item[(a)] For $q\le2$ we have $A^p_\alpha\subset H^q$ if and only if $\alpha\le-1$.
For $q>2$ we have $A^p_\alpha\subset H^q$ if and only if $\alpha<-1$.
\item[(b)] For $q<2$ we have $H^p\subset A^q_\alpha$ if and only if $\alpha>-1$.
For $q\ge2$ we have $H^p\subset A^q_\alpha$ if and only if $\alpha\ge-1$.
\item[(c)] Any embedding in (a) and (b) is compact if and only if the inequality
involving $\alpha$ is strict.
\end{enumerate}
\end{thm}

\begin{proof}
The case $p=q$ follows from Theorem \ref{13}.

For any $q>0$ and $\alpha\in\R$ we let $m$ denote the smallest non-negative integer
such that $mq+\alpha>-1$. Then by the closed-graph theorem, the inclusion
$H^p\subset A^q_\alpha$ is equivalent to the existence of a positive constant $C$
such that
$$\left[\inb|R^mf(z)|^q(1-|z|^2)^{mq+\alpha}\,dv(z)\right]^{\frac1q}\le C\|f\|_{H^p}$$
for all $f\in H^p$. This is exactly the kind of estimates in Theorem~\ref{14}, so
we must consider the function
$$g(z)=\frac1{(1-|z|^2)^{mq+n}}\int_{D(z,1)}(1-|w|^2)^{mq+\alpha}\,dv(w).$$
It is well known (see \cite{Zhu1} for example) that $v(D(z,1))\sim(1-|z|^2)^{n+1}$ and
$1-|w|^2\sim 1-|z|^2$ for $w\in D(z,1)$. Thus we may as well assume that
$$g(z)=\frac{(1-|z|^2)^{mq+\alpha}(1-|z|^2)^{n+1}}{(1-|z|^2)^{mq+n}}
=(1-|z|^2)^{1+\alpha}.$$

By symmetry, the functions $A_r(g)$ and $A_\infty(g)$ are obviously constant
on $\sn$. For any $\zeta\in\sn$ we have
$$A_\infty(g)(\zeta)=\sup_{z\in\Gamma(\zeta)}(1-|z|^2)^{1+\alpha}.$$
Since $\Gamma(\zeta)$ is a region in $\bn$ that ``approaches'' the boundary point
$\zeta$, it is clear that $A_\infty(g)(\zeta)$ is a finite number if and only if
$1+\alpha\ge0$ or $\alpha\ge-1$. Therefore, the constant function $A_\infty(g)$
belongs to $L^{p/(p-q)}(\sn, d\sigma)$ if and only if $\alpha\ge-1$.

Similarly, for $q<2$ and $r=2/(2-q)$, we have
\begin{align*}
A_r(g)(\zeta)&=\left[\int_{\Gamma(\zeta)}
\frac{(1-|z|^2)^{(1+\alpha)r}}{(1-|z|^2)^{n+1}}\,dv(z)\right]^{\frac1r}\\
&=\left[\int_{\Gamma(\zeta)}\frac{dv(z)}{(1-|z|^2)^{n+1-(1+\alpha)r}}\right]^{\frac1r}.
\end{align*}
By Lemma \ref{15}, we have $A_r(g)(\zeta)<\infty$ if and only if $(1+\alpha)r>0$
or $\alpha>-1$. Consequently, the constant function $A_r(g)$ belongs to
$L^{p/(p-q)}(\sn, d\sigma)$ if and only if $\alpha>-1$. This completes
the proof of part (b).

To prove part (a) of Theorem B, we first assume that $q\le2$ and
$A^p_\alpha\subset H^q$. By integration in polar coordinates, we have
$H^q\subset A^q_\beta$ for any $\beta>-1$. Thus $A^p_\alpha\subset A^q_\beta$
for all $\beta>-1$. By Theorem C, we must have
$$\frac{1+\alpha}p<\frac{1+\beta}q.$$
Let $\beta\to-1^+$. We obtain $\alpha\le-1$.

On the other hand, if $q\le2$ and $\alpha\le-1$, then $A^p_\alpha\subset H^q$.
In fact, for $\alpha<-1$, it follows from Lemma \ref{12} that $A^p_\alpha\subset H^p
\subset H^q$. When $\alpha=-1$ and $p\le 2$, we have $A^p_{-1}\subset H^p
\subset H^q$. If $\alpha=-1$ and $q\le 2<p$, then $A^p_{-1}\subset A^2_{-1}
=H^2\subset H^q$. This completes the proof of part (a) in the case $q\le2$.

Next we assume that $q>2$. If $\alpha<-1$, it follows from Lemmas \ref{1} and
\ref{12} that $A^p_\alpha\subset A^q_\alpha\subset H^q$. On the other hand, if
$A^p_\alpha\subset H^q$, it follows from Lemma \ref{11} that $A^p_\alpha\subset
A^q_{-1}$. By Theorem C, we must have $\alpha<-1$. This proves part (a) in
the case $q>2$.

Finally, part (c) follows from Lemmas \ref{1}, \ref{2}, and \ref{11}.
The proof of the theorem is now complete.
\end{proof}

\section{Tight fittings}

When we have one analytic function space embedded/contained in another analytic
function space, the closed-graph theorem then implies that the inclusion map is 
bounded. For example, if $H^p\subset A^q_\alpha$, then there exists a positive 
constant $C$ such that $\|f\|_{A^q_\alpha}\le C\|f\|_{H^p}$ for all $f\in H^p$. 
The problem of determining the best constant (namely, the norm of the inclusion 
operator) is very interesting. The problem is often difficult and it has been studied 
in the literature by many authors.

Of particular interest is the case when the norm of the inclusion operator equals $1$
and when the involved spaces have norms defined naturally in a canonical way. We
mention here a few well-known examples. Note that $A^p$ means $A^p_\alpha$
for $\alpha=0$, and all examples below are for the open unit disc.

\begin{example}\label{17}
A classical theorem of Carleman's states that we have $H^p\subset A^{2p}$ with
$$\|f\|_{A^{2p}}\le\|f\|_{H^p},\qquad f\in H^p,$$
where $0<p<\infty$. Furthermore, equality holds if and only if $f$ is of the form
$f(z)=c(1-az)^{-2/p}$, where $c\in\C$ and $a\in\D$. See \cite{C}. When $p=1$,
Carleman's inequality above is closely related to the classical isoperimetric
inequality on the complex plane.
\end{example}

\begin{example}\label{18}
In 1987, Burbea \cite{B} generalized Carleman's inequality as follows: $H^p\subset
A^{kp}_{k-2}$ and the inclusion is contractive, where $0<p<\infty$ and $k\ge2$
is an integer.
\end{example}

\begin{example}\label{19}
In 2014, Pavlovi\'c made the following conjecture in \cite{Pav}: for any $p>0$ and
$\alpha>-1$, the Hardy space $H^p$ is contractively contained in the weighted
Bergman space $A^{(\alpha+2)p}_\alpha$. Pavlovi\'c's conjecture covers Burbea's
result as a special case.
\end{example}

\begin{example}\label{20}
Bayart-Brevig-Haimi-Ortega-Cerd\'a-Perfect \cite{BBHOP}
proved in 2019 that, for any $\alpha>(1+\sqrt{17})/4$, the weighted Bergman
space $A^p_{\alpha-2}$ is contractively embedded in the weighted
Bergman space $A^{p(\alpha+1)/\alpha}_{\alpha-1}$.
\end{example}

Similar embedding problems were also considered by Bondarenko, Brevig,
Kalaj, Llinares, Ortega-Cerd\'a, Saksman, Seip, and Zhao in
\cite{BBSSZ,BOSZ,Ka,LL0, LL}. The most complete results in this direction,
however, were obtained by Kulikov \cite{K} in 2022, which settles all the
previous conjectures about contractive embeddings for Hardy and Bergman
spaces on the unit disc.

All examples above, including Kulikov's results, involve inclusions that are NOT
compact. Of course, there are many embeddings that are obviously contractive,
for example, $H^p\subset H^q$ for $p>q$, $H^p\subset A^p_\alpha$ for
$\alpha>-1$, and $A^p_\alpha\subset A^q_\alpha$ for $p>q$ and
$\alpha>-1$. All these obvious embeddings are compact and not interesting.
This motivates our next definition.

\begin{definition}\label{21}
An embedding $X\subset Y$ of analytic function spaces is called a {\it tight fitting}
if the inclusion is proper, contractive, and non-compact.
\end{definition}

The following conjecture attempts to completely describe all possible tight fittings
for Hardy spaces and weighted Bergman spaces on the unit ball.

\begin{conj}\label{22}
Suppose $0<p,q<\infty$ and $\alpha,\beta>-1$. Then
\begin{enumerate}
\item[(a)] $A^p_\alpha\subset A^q_\beta$ is a tight fitting if and only if
$$p<q\quad{\rm and}\quad (n+1+\alpha)q=(n+1+\beta)p.$$
Furthermore, $\|f\|_{A^p_\alpha}=\|f\|_{A^q_\beta}$
if and only if $f$ is of the form
$$f(z)=\frac c{(1-\langle z,a\rangle)^{2(n+1+\alpha)/p}},$$
where $c\in\C$ and $a\in\bn$.
\item[(b)] $H^p\subset A^q_\alpha$ is a tight fitting if and only if
$$p<q\quad{\rm and}\quad nq=(n+1+\alpha)p.$$
Furthermore, $\|f\|_{A^q_\alpha}=\|f\|_{H^p}$ if and only if $f$ is of the form
$$f(z)=\frac c{(1-\langle z,a\rangle)^{2n/p}},$$
where $c\in\C$ and $a\in\bn$.
\end{enumerate}
\end{conj}

Together with our compactness results in this paper, the conjecture above has
been proved by Kulikov \cite{K} when $n=1$, except the description of extremal
functions, namely, functions $f$ satisfying $\|f\|_{A^p_\alpha}=\|f\|_{A^q_\beta}$
or $\|f\|_{H^p}=\|f\|_{A^q_\alpha}$. It was observed in \cite{K} (and it is easy
to verify directly for all $n\ge1$) that the kernel functions in Conjecture~\ref{22}
are indeed extremal functions. We conjecture here that they are the only ones.
This complete determination of extremal functions was known in certain special
cases, for example in Carleman's theorem.

Kulikov's results in \cite{K} has been generalized to the open unit ball by
Li-Su \cite{LS}, with one additional technical assumption about the hyperbolic
geometry of $\bn$. As of this writing, it is not clear if that technical difficulty
can be overcome somehow.

Note that in the conjecture above we made the additional assumptions that
$\alpha>-1$ and $\beta>-1$. We believe that, at least for now, this is necessary.
Otherwise, the definition of norms for the corresponding weighted Bergman spaces
has to involve derivatives of  functions, which would result in norms that are not
``natural'' or ``canonical''. When we consider contractive embeddings, the function
norms should be natural or somehow canonical, because we want the operator norm
of the inclusion map to be {\it exactly} $1$.

In particular, with the understanding in the previous paragraph, there is no tight
fitting of the form $A^p_\alpha\subset H^q$. In fact, the containment is clearly
impossible for $p\le q$ and $\alpha>-1$. Also, by Theorem B, if $p>q$, then the
containment $A^p_\alpha\subset H^q$ would imply that $\alpha\le-1$.

\end{document}